\documentclass[12pt]{amsart}

\usepackage{amssymb}
\usepackage{bm}
\usepackage{graphicx}
\usepackage{psfrag}
\usepackage{color}
\usepackage{enumerate}
\usepackage{url}

\usepackage{algpseudocode}

\usepackage{mathtools}

\usepackage{xy}
\input xy
\xyoption{all}

\newcommand{\excise}[1]{}

\newtheorem{thm}{Theorem}[section]
\newtheorem{lemma}[thm]{Lemma}

\newtheorem{cor}[thm]{Corollary}
\newtheorem{prop}[thm]{Proposition}

\theoremstyle{definition}

\newtheorem{remark}[thm]{Remark}
\newtheorem{defn}[thm]{Definition}

\numberwithin{equation}{section}

\newcommand{\ring}[1]{\ensuremath{\mathbb{#1}}}

\renewcommand\>{\rangle}
\newcommand\<{\langle}

\newcommand\RR{\ring{R}}

\newcommand\ZZ{\ring{Z}}

\DeclareMathOperator\Ap{Ap} 

\begin{document}

\mbox{}
\title{Ap\'ery sets of shifted numerical monoids}

\author{Christopher O'Neill}
\address{Mathematics Department\\University of California Davis\\Davis, CA 95616}
\email{coneill@math.ucdavis.edu}

\author{Roberto Pelayo}
\address{Mathematics Department\\University of Hawai`i at Hilo\\Hilo, HI 96720}
\email{robertop@hawaii.edu}



\date{\today}

\begin{abstract}
A numerical monoid is an additive submonoid of the non-negative integers.  Given a numerical monoid $S$, consider the family of ``shifted'' monoids $M_n$ obtained by adding $n$ to each generator of $S$.  In this paper, we characterize the Ap\'ery set of $M_n$ in terms of the Ap\'ery set of the base monoid $S$ when $n$ is sufficiently large.  We give a highly efficient algorithm for computing the Ap\'ery set of $M_n$ in this case, and prove that several numerical monoid invariants, such as the genus and Frobenius number, are eventually quasipolynomial as a function of $n$.  
\end{abstract}

\maketitle


\section{Introduction} \label{sec:intro}


The factorization theory of numerical monoids -- co-finite, additive submonoids of the non-negative integers -- has enjoyed much recent attention; in particular, invariants such as the minimum factorization length, delta set, and $\omega$-primality have been studied in much detail \cite{numericalsurvey}.  These measures of non-unique factorization for individual elements in a numerical monoid all exhibit a common feature: eventual quasipolynomial behavior.  In many cases, this eventual behavior is periodic (i.e. quasiconstant) or quasilinear, and this pattern always holds after some initial ``noise'' for small monoid elements.  

In this paper, we describe quasipolynomial behavior of certain numerical monoid invariants over parameterized monoid families.  Unlike previous papers, which studied how factorization invariants change element-by-element (e.g., minimum factorization length \cite{elastsets}), we investigate how a monoid's properties change as the generators vary by a shift parameter.  More specifically, we study ``shifted'' numerical monoids of the form 
$$M_n =  \<n, n+r_1, \ldots, n + r_k\>$$
for $r_1 < \cdots < r_k$, and find explicit relationships between the Frobenius number, genus, type, and other properties of $M_n$ and $M_{n + r_k}$ when $n > r_k^2$.  As with the previous element-wise investigations of invariant values, our monoid-wise analysis reveals eventual quasipolynomial behavior, this time with respect to the shift parameter $n$.  

The main result of this paper is Theorem~\ref{t:aperyshifted}, which characterizes the Ap\'ery set of~$M_n$ (Definition~\ref{d:apery}) for large $n$ in terms of the Ap\'ery set of the monoid $S = \<r_1, \ldots, r_k\>$ at the base of the shifted family.  Ap\'ery sets are non-minimal generating sets that concisely encapsulate much of the underlying monoid structure, and many properties of interest can be recovered directly and efficiently from the Ap\'ery set, making it a sort of ``one stop shop'' for computation.  We utilize these connections in Section~\ref{sec:applications} to derive relationships between properties of $M_n$ and $M_{n + r_k}$ when $n$ is sufficiently large.  

One of the main consequences of our results pertains to computation.  Under our definition of $M_n$ above, every numerical monoid is a member of some shifted family of numerical monoids.  While Ap\'ery sets of numerical monoids (and many of the properties derived from them) are generally more difficult to compute when the minimal generators are large, our results give a way to more efficiently perform these computations by instead computing them for the numerical monoid $S$, which has both smaller and fewer generators than $M_n$.  In fact, one surprising artifact of the algorithm described in Remark~\ref{r:aperyruntimes} is that, in a shifted family $\{M_n\}$ of numerical monoids, the computation of the Ap\'ery set of $M_n$ for $n > r_k^2$ is typically significantly faster than for $M_n$ with $n \leq r_k^2$, even though the former has larger generators.   We discuss this and further computational consequences in Remark~\ref{r:computeshifted}, including implementation of our algorithm in the popular \texttt{GAP} package \texttt{numericalsgps} \cite{numericalsgpsgap}.


\section{Background}
\label{sec:background}

In this section, we recall several definitions and results used in this paper.  For more background on numerical monoids, we direct the reader to~\cite{numerical}.  

\begin{defn}\label{d:numericalmonoid}
A \emph{numerical monoid} $M$ is an additive submonoid of $\ZZ_{\ge 0}$.  When we write $M = \<n_1, \ldots, n_k\>$, we assume $n_1 < \cdots < n_k$.  We say $M$ is \emph{primitive} if $\gcd(n_1, \ldots, n_k) = 1$.  
A \emph{factorization} of an element $a \in M$ is an expression 
$$a = z_1n_1 + \cdots + z_kn_k$$
of $a$ as a sum of generators of $M$, which we often represent with the integer tuple $\vec z = (z_1, \ldots, z_k) \in \ZZ_{\ge 0}^k$.  The \emph{length} of a factorization $\vec z$ of $a$ is the number
$$|\vec z| = z_1 + \cdots + z_k$$
of generators appearing in $z$.  The set of factorizations of $a$ is denoted $\mathsf Z_M(a) \subset \ZZ_{\ge 0}^k$, and the set of factorization lengths is denoted $\mathsf L_M(a) \subset \ZZ_{\ge 0}$.  
\end{defn}

\begin{defn}\label{d:apery}
Let $M$ be a numerical monoid.  Define the \emph{Ap\'ery set} of $x \in M$ as
$$\Ap(M,x) = \{m \in M : m-x \in \ZZ \setminus M\}$$
and the \emph{Ap\'ery set} of $M$ as $\Ap(M) = \Ap(M;n_1)$, where $n_1$ is the smallest nonzero element of $M$.  
Note that under this definition, $\left|\Ap(M;x)\right| = x/\gcd(M)$.  
\end{defn}

\begin{defn}\label{d:genus}
Suppose $M$ is a numerical monoid with $\gcd(M) = 1$.  The \emph{genus} of $M$ is the number $g(M) = |\ZZ_{\ge 0} \setminus M|$ of positive integers lying outside of $M$.  The largest integer $F(M) = \max(\ZZ_{\ge 0} \setminus M)$ outside of $M$ is the \emph{Frobenius number} of $M$.  For a non-primitive monoid $T = dM$ with $d \ge 1$, define $g(T) = d \cdot g(M)$ and $F(T) = d \cdot F(M)$.  
\end{defn}

Theorem~\ref{t:minlenquasi} appeared in \cite{elastsets} for primitive, minimally generated numerical monoids, and in \cite{shiftyminpres} for general numerical monoids.  We state the latter version here.  

\begin{thm}[{\cite{elastsets,shiftyminpres}}]\label{t:minlenquasi}
Suppose $M = \<n_1, \ldots, n_k\>$ is a numerical monoid.  The function $\mathsf m:M \to \ZZ_{\ge 0}$ sending each $a \in M$ to its shortest factorization length satisfies
$$\mathsf m(a + n_k) = \mathsf m(a) + 1$$
for all $a > n_{k-1}n_k$.  
\end{thm}

\subsection*{Notation}
Through the remainder of this paper, $r_1 < \cdots < r_k$ and $n$ are non-negative integers, $d = \gcd(r_1, \ldots, r_k)$, and 
$$S = \<r_1, \ldots, r_k\> \qquad \text{ and } \qquad M_n = \<n, n + r_1, \ldots, n + r_k\>$$
denote additive submonoids of $\ZZ_{\ge 0}$.  Unless otherwise stated, we assume $n > r_k$ and $\gcd(n,d) = 1$ so $M_n$ is primitive and minimally generated as written, but we do \textbf{not} make any such assumptions on $S$.  Note that choosing $n$ as the first generator of $M_n$ ensures that every numerical monoid falls into exactly one shifted family.

\section{Ap\'ery sets of shifted numerical monoids}
\label{sec:aperysets}

The main results in this section are Theorem~\ref{t:aperyshifted}, which expresses $\Ap(M_n)$ in terms of $\Ap(S;n)$ for $n$ sufficiently large, and Proposition~\ref{p:largeapery}, which characterizes $\Ap(S;n)$ for large $n$ in terms of $\ZZ_{\ge 0} \setminus S$.  In addition to the numerous consequences in Section~\ref{sec:applications}, these two results yield an algorithm to compute $\Ap(M_n)$ for large $n$; see Remark~\ref{r:aperyruntimes}.


\begin{lemma}[{\cite[Theorem~3.4]{shiftyminpres}}]\label{l:mesalemma}
Suppose $n > r_k^2$.  If $\vec a$ and $\vec b$ are factorizations of an element $m \in M_n$ with $|\vec a| < |\vec b|$, then some factorization $\vec b'$ with $|\vec b| = |\vec b'|$ has $b'_0 > 0$.  
\end{lemma}



\begin{remark}\label{r:homogeneous}
Lemma~\ref{l:mesalemma} was also proven in \cite[Corollary~3.6]{vu14} and subsequently improved in \cite[Theorem~6.3]{homogeneousnumerical}, both with strictly higher bounds on the starting value of~$n$.  The latter source proved that such numerical monoids are \emph{homogeneous}, meaning every element of the Ap\'ery set has a unique factorization length.  This property appears as part of Theorem~\ref{t:aperyshifted} along with a characterization of the unique length in terms of~$S$.  
\end{remark}

\begin{thm}\label{t:aperyshifted}
If $n > r_k^2$ and $dn \in S$, then 
$$\Ap(M_n) = \{i + \mathsf m_S(i) \cdot n \mid i \in \Ap(S;dn)\}$$
where $\mathsf m_S$ denotes the minimum factorization length in $S$.  Moreover, we have 
$$\mathsf L_{M_n}(i + \mathsf m_S(i) \cdot n) = \{\mathsf m_S(i)\}$$
for each $i \in \Ap(S;dn)$.  
\end{thm}

\begin{proof}
Let $A = \{i + \mathsf m_S(i) \cdot n \mid i \in \Ap(S;dn)\}$.  Each element of $\Ap(S;dn)$ is distinct modulo $n$, since each element of $\{i/d \mid i \in \Ap(S;dn)\}$ is distinct modulo $n$ and $\gcd(n,d) = 1$.  As such, each element of $A$ is distinct modulo $n$, since
$$i + \mathsf m_S(i) \cdot n \equiv i \bmod n$$
for $i \in \Ap(S;dn)$.  This implies $|A| = n$, so it suffices to show $A \subseteq \Ap(M_n)$.  

Fix $i \in \Ap(S;dn)$, and let $a = i + \mathsf m_S(i) \cdot n$.  If $\vec s \in \mathsf Z_S(i)$ has minimal length, then
$$a = i + \mathsf m_S(i) \cdot n = \sum_{i = 1}^k s_ir_i + |\vec s| \cdot n = \sum_{i = 1}^k s_i(n + r_i),$$
meaning $a \in S$.  In this way, each minimal length factorization $\vec s$ for $i \in S$ corresponds to a factorization of $a \in M_n$ with first component zero.  More generally, for each $\ell \ge 0$, there is a natural bijection 
$$\begin{array}{rcl}
\{\vec z \in \mathsf Z_{M_n}(a) : |\vec z| = \ell\} &\longrightarrow& \{\vec s \in \mathsf Z_S(a - \ell n) : |\vec s| \le \ell\} \\
(z_0, z_1, \ldots, z_k) &\longmapsto& (z_1, \ldots, z_k)
\end{array}$$
between the factorizations of $a \in M_n$ of length $\ell$ and the factorizations of $a - \ell n \in S$ of length at most $\ell$, obtained by writing 
$$a = z_0n + \sum_{i = 1}^k z_i(n + r_i) = \ell n + \sum_{i = 1}^k z_ir_i$$
and subsequently solving for $a - \ell n$.  

Now, since $a - \mathsf m_S(i) \cdot n = i \in \Ap(S;dn)$, whenever $\ell > \mathsf m_S(i)$ we have 
$$a - \ell n = (a - \mathsf m_S(i) \cdot n) - (\ell - \mathsf m_S(i))n = i - (\ell - \mathsf m_S(i))n \notin S,$$
so $a$ has no factorizations in $M_n$ of length $\ell$.  Moreover, $a$ cannot have any factorizations in $M_n$ with length strictly less than $\mathsf m_S(i)$ since Lemma~\ref{l:mesalemma} would force some factorization of length $\mathsf m_S(i)$ to have nonzero first coordinate.  
Putting these together, we see $\mathsf L_{M_n}(a) = \{\mathsf m_S(i)\}$, meaning every factorization of $a$ in $M_n$ has first coordinate zero.  As such, we conclude $a \in \Ap(M_n)$.  
\end{proof}

\begin{prop}\label{p:largeapery}
If $dn \in S$ and $dn > F(S)$, then $\Ap(S;dn) = \{a_0, \ldots, a_{n-1}\}$, where
$$a_i = \left\{\begin{array}{ll}
di & \text{ if } di \in S \\
di + dn & \text{ if } di \notin S
\end{array}\right.$$
and $d = \gcd(S)$.  In particular, this holds whenever $n > r_k^2$ as in Theorem~\ref{t:aperyshifted}.  
\end{prop}

\begin{proof}
Clearly $a_i \in \Ap(S;dn)$ for each $i \le n-1$.  Moreover, the values $a_0, \ldots, a_{n-1}$ are distinct modulo $n$, so the claimed equality holds.  
\end{proof}

\begin{remark}\label{r:aperyruntimes}
Theorem~\ref{t:aperyshifted} and Proposition~\ref{p:largeapery} yield an algorithm to compute the Ap\'ery set of numerical monoids $M = \<n_1, \ldots,  n_k\>$ that are \emph{sufficiently shifted} (that is, if $n_1 > (n_k - n_1)^2$) by first computing the Ap\'ery set for $S = \<n_i - n_1 : 2 \le i \le k\>$.  
Table~\ref{tb:aperyruntimes} compares the runtime of this algorithm to the one currently implemented in the \texttt{GAP} package \texttt{numericalsgps} \cite{numericalsgpsgap}.  The strict decrease in runtime halfway down the last column corresponds to values of $n$ where $n > r_k^2 = 20^2 = 400$, so we may use Proposition~\ref{p:largeapery} to express $\Ap(S;dn)$ directly in terms of the gaps of $S$.  This avoids performing extra modular arithmetic computations for each Ap\'ery set element, as is normally required to compute $\Ap(S;dn)$ from $\Ap(S)$.  

An implementation will appear in the next release of the \texttt{numericalsgps} package, and will not require any special function calls.  In particular, computing the Ap\'ery set of a monoid $M$ that is sufficiently shifted will automatically use the improved Ap\'ery set algorithm, and resort to the existing algorithm in all other cases.  
\end{remark}

\begin{table}
\begin{tabular}{|l|l|l|l|l|}
$n$ & $M_n$ & \texttt{GAP} \cite{numericalsgpsgap} & Remark~\ref{r:aperyruntimes} \\
\hline
$50$ & $\<50,56,59,70\>$ & 1 ms & 1 ms \\
$200$ & $\<200, 206, 209, 220\>$ & 30 ms & 30 ms \\
$400$ & $\<400, 406, 409, 420\>$ & 170 ms & 170 ms \\
$1000$ & $\<1000, 1006, 1009, 1020\>$ & 3 sec & 1 ms \\
$5000$ & $\<5000, 5006, 5009, 5020\>$ & 17 min & 1 ms \\ 
$10000$ & $\<10000, 10006, 10009, 10020\>$ & 3.6 hr & 1 ms \\ 
\end{tabular}
\medskip
\caption{Runtime comparison for computing Ap\'ery sets of the numerical monoids $M_n$ with $S = \<6,9,20\>$.  All computations performed using \texttt{GAP} and the package \texttt{numericalsgps} \cite{numericalsgpsgap}.}
\label{tb:aperyruntimes}
\end{table}

\begin{remark}\label{r:frequency}
Much of the literature on numerical monoids centers around especially ``nice'' families, e.g.\ those that are symmetric~\cite{kunz70}, complete intersection~\cite{completeintersection}, balanced~\cite{balancednumerical}, telescopic~\cite{telescopic}, Arf~\cite{arfnumerical}, or generated by an arithmetic~\cite[Chapter~4]{numerical} or compound~\cite{compseqs} sequence, among many others~\cite{numerical}.  
Such families are usually chosen so that a precise description of some of their properties can be obtained, thereby lending insight into the (often provably intractable) general case.  
Remark~\ref{r:aperyruntimes} identifies \emph{sufficiently shifted} numerical monoids as a new family to add to the list.  
%
\end{remark}

\section{Applications}
\label{sec:applications}

As Ap\'ery sets can be used to easily compute other numerical monoid invariants, the results of Section~\ref{sec:aperysets} can be applied to provide quick computations of and structural results for the Frobenius number, genus (Definition~\ref{d:genus}), and type (Definition~\ref{d:pseudofrob}) of $M_n$ for $n$ sufficiently large.  In particular, we show that each of these are eventually quasipolynomial functions of $n$  (Corollaries~\ref{c:frobshifted}, \ref{c:genusshifted}, and~\ref{c:typeshifted}, respectively).  

\begin{defn}\label{d:quasipolynomial}
A function $f:\ZZ \to \RR$ is an \emph{$r$-quasipolynomial} of degree $\alpha$ if 
$$f(n) = a_\alpha(n)n^\alpha + \cdots + a_1(n)n + a_0(n)$$
for periodic functions $a_0, \ldots, a_\alpha$, whose periods all divide $r$, with $a_\alpha$ not identically~0.  We say $f$ is \emph{eventually quasipolynomial} if the above equality holds for all $n \gg 0$.  
\end{defn}

\begin{cor}\label{c:genusshifted}
For $n > r_k^2$, the function $n \mapsto g(M_n)$ is $r_k$-quasiquadratic in $n$.  
\end{cor}

\begin{proof}
By counting the elements of $\ZZ_{\ge 0} \setminus M_n$ modulo $n$, we can write
$$\begin{array}{r@{}c@{}l}
g(M_n) &{}={}& \displaystyle\sum_{a \in \Ap(M_n)} \left\lfloor \frac{a}{n} \right\rfloor
\end{array}$$
Applying Theorem~\ref{t:aperyshifted} and Proposition~\ref{p:largeapery}, a simple calculation shows that
$$\begin{array}{r@{}c@{}l}
g(M_n) 
&{}={}& \displaystyle\sum_{i \in \Ap(S;dn)} \left\lfloor \frac{i}{n} \right\rfloor + \sum_{i \in \Ap(S;dn)} \mathsf m_S(i) \\
&{}={}& \displaystyle\sum_{i = 0}^{n-1} \left\lfloor \frac{di}{n} \right\rfloor + d \cdot g(S) + \sum_{\substack{i < n \\ di \in S}} \mathsf m_S(di) + \sum_{\substack{i \ge 0 \\ di \notin S}} \mathsf m_S(di + dn).
\end{array}$$
Each of the four terms in the above expression is eventually quasipolynomial in $n$.  Indeed, the first term is $d$-quasilinear in $n$, the second term is independent of $n$, and Theorem~\ref{t:minlenquasi} implies the third and fourth terms are eventually $r_k$-quasiquadratic and $r_k$-quasilinear in $n$, respectively.  This completes the proof.  
\end{proof}

Corollary~\ref{c:frobshifted} is known in more general contexts \cite{rouneparametricfrob,shenparametricfrob}.  The proof given here yields a fast algorithm for computing the Frobenius number of $M_n$ for $n > r_k^2$; see Remark~\ref{r:computeshifted}.  

\begin{cor}\label{c:frobshifted}
For $n > r_k^2$, the function $n \to F(M_n)$ is $r_k$-quasiquadratic in $n$.  
\end{cor}

\begin{proof}
Let $a$ denote the element of $\Ap(S;dn)$ for which $\mathsf m_S(-)$ is maximal.  Theorem~\ref{t:aperyshifted} and Proposition~\ref{p:largeapery} imply
\begin{center}
$\begin{array}{r@{}c@{}l}
F(M_n) &{}={}& \max(\Ap(M_n)) - n = a - n + \mathsf m_S(a) \cdot n.
\end{array}$
\end{center}
Theorem~\ref{t:minlenquasi} and Proposition~\ref{p:largeapery} together imply $a + r_k$ is the element of $\Ap(S;dn + r_k)$ for which $\mathsf m_S(-)$ is maximal, and quasilinearity of $\mathsf m_S(-)$ completes the proof.  
%
%
\end{proof}

\begin{remark}\label{r:wilfconjecture}
Wilf's conjecture \cite{wilfconjecture}, a famously open problem in the numerical monoids literature, states that for any primitive numerical monoid $S = \<r_1, \ldots, r_k\>$, 
$$F(S) + 1 \le k(F(S) - g(S))$$
To date, Wilf's conjecture has only been proven in a handful of special cases, and remains open in general.  Corollary~\ref{c:wilfshifted} adds monoids of the form $M_n$ for $n > r_k^2$ to the list of special cases in which Wilf's conjecture is known to hold.  
\end{remark}

\begin{defn}[{\cite{questioneliahouwilf}}]\label{d:wilfnumber}
The \emph{Wilf number} $W(S)$ of a numerical monoid $S$ with embedding dimension $k$ is given by 
$$W(S) = k(F(S) - g(S)) - (F(S) + 1).$$
\end{defn}

\begin{cor}\label{c:wilfshifted}
For $n > r_k^2$, the function $n \mapsto W(M_n)$ is $r_k$-quasiquadratic in $n$.  In~particular, $M_n$ satisfies Wilf's conjecture for $n > r_k^2$.  
\end{cor}

\begin{proof}
Apply Corollaries~\ref{c:genusshifted} and~\ref{c:frobshifted}, and note that the quadratic coefficients of the maps $n \mapsto F(M_n)$ and $n \mapsto g(M_n)$ are constants $d/r_k$ and $d/2r_k$, respectively.  
\end{proof}

We next examine the pseudo-Frobenius numbers of $M_n$.  

\begin{defn}\label{d:pseudofrob}
An integer $m \ge 0$ is a \emph{pseudo-Frobenius number} of a numerical monoid $S$ if $m \notin S$ but $m + n \in S$ for all positive $n \in S$.  Denote by $PF(S)$ the set of pseudo-Frobenius numbers of $S$, and by $t(S) = |PF(M_n)|$ the \emph{type} of $S$.  
\end{defn}

\begin{thm}\label{t:pseudofrobshifted}
Given $n \in \ZZ_{\ge 0}$, let $P_n$ denote the set
$$P_n = \{i \in \Ap(S;dn) : a \equiv i \bmod n \text{ for some } a \in PF(M_n)\}.$$
For $n > r_k^2$, the map $P_n \to P_{n + r_k}$ given by
$$\begin{array}{rcl}
i
&\mapsto&
\left\{\begin{array}{ll}
i & \text{if } i \le dn \\
i + r_k & \text{if } i > dn
\end{array}\right.
\end{array}
$$
is a bijection.  In particular, there is a bijection $PF(M_n) \to PF(M_{n + r_k})$.  
\end{thm}

\begin{proof}
Fix $i \in \Ap(S;dn)$ and write $a = i + \mathsf m_S(i)n \in \Ap(M_n;n)$.  First, if $i \le dn$, then~$i \in \Ap(S;dn + dr_k)$ by Proposition~\ref{p:largeapery}, so by Theorem~\ref{t:aperyshifted}, 
$$a' = i + \mathsf m_S(i)(n + r_k) = a + \mathsf m_S(i)r_k \in \Ap(M_{n + r_k};n + r_k).$$
Notice that if $a + r_j \in M_n$, then the bijection in the proof of Theorem~\ref{t:aperyshifted} implies $a + r_j$ has a factorization of length $\mathsf m_S(i)$ in $M_n$.  As such, for each $j$ we have $a' + r_j \in M_{n + r_k}$ if and only if $a + r_j \in M_n$.  On the other hand, if $i > dn$, then $i + r_k \in \Ap(S;dn + dr_k)$ by Proposition~\ref{p:largeapery}, so by Theorem~\ref{t:aperyshifted}, 
$$a' = i + r_k + \mathsf m_S(i + r_k)(n + r_k) = a + (n + r_k) + (\mathsf m_S(i) + 1)r_k \in \Ap(M_{n + r_k};n + r_k).$$
Once again, $a' + r_j \in M_{n + r_k}$ if and only if $a + r_j \in M_n$ for each $j$, thus proving the first claim.  The second claim is obtained using the composition
$$PF(M_n) \longrightarrow P_n \longrightarrow P_{n + r_k} \longrightarrow PF(M_{n + r_k})$$
of the above map with two bijections obtained via reduction modulo $n$.  
\end{proof}

\begin{cor}\label{c:typeshifted}
The function $n \mapsto t(M_n)$ is eventually $r_k$-periodic.  In particular, $M_n$~is (pseudo)symmetric if and only if $M_{n+r_k}$ is (pseudo)symmetric. 
\end{cor}

\begin{proof}
Apply Theorem~\ref{t:pseudofrobshifted} and \cite[Corollaries~3.11 and~3.16]{numerical}.  
\end{proof}



\begin{remark}\label{r:computeshifted}
Each of the quantities and properties discussed in this section are usually computed for a general numerical monoid $M$ by first computing an Ap\'ery~set.  Indeed, computing each of these quantities for the monoids in Table~\ref{tb:aperyruntimes} takes only slightly longer than the corresponding Ap\'ery set runtime.  As~such, computing these values using the algorithm discussed in Remark~\ref{r:aperyruntimes} is also significantly faster for $n > r_k^2$.  
\end{remark}


\end{document}